\documentclass[12pt,english]{amsart}
\usepackage[T1]{fontenc}
\usepackage[latin9]{inputenc}
\usepackage{geometry}
\usepackage{babel}
\usepackage{array}
\usepackage{amsthm}
\usepackage{amstext}
\usepackage{amssymb}
\usepackage{graphicx}
\usepackage[unicode=true,pdfusetitle,
 bookmarks=true,bookmarksnumbered=false,bookmarksopen=false,
 breaklinks=false,pdfborder={0 0 1},backref=false,colorlinks=false]
 {hyperref}

\makeatletter

\providecommand{\tabularnewline}{\\}

\numberwithin{equation}{section}
\theoremstyle{plain}
\newtheorem{thm}{\protect\theoremname}
  \theoremstyle{definition}
  \newtheorem{defn}[thm]{\protect\definitionname}
  \theoremstyle{plain}
  \newtheorem{lem}[thm]{\protect\lemmaname}

\usepackage{psfrag}

\newlength{\myArraycolsep}
\usepackage{subfig}
\usepackage{chngcntr}
\counterwithout{figure}{section}

\@ifundefined{showcaptionsetup}{}{%
 \PassOptionsToPackage{caption=false}{subfig}}
\usepackage{subfig}
\makeatother

  \providecommand{\definitionname}{Definition}
  \providecommand{\lemmaname}{Lemma}
\providecommand{\theoremname}{Theorem}

\begin{document}

\title{Line graphs and the transplantation method}

\author{Peter Herbrich}
\begin{abstract}
We study isospectrality for mixed Dirichlet-Neumann boundary conditions,
and extend the previously derived graph-theoretic formulation of the
transplantation method. Led by the theory of Brownian motion, we introduce
vertex-colored and edge-colored line graphs that give rise to block
diagonal transplantation matrices. In particular, we rephrase the
transplantation method in terms of representations of free semigroups,
and provide a method for generating adjacency cospectral weighted
directed graphs.
\end{abstract}

\address{Department of Mathematics, Dartmouth College, Hanover, NH, USA}

\email{peter.herbrich@dartmouth.edu}

\maketitle
\thispagestyle{empty}

\setlength{\skip\footins}{8mm}

\let\thefootnote\relax\footnote{\textbf{Acknowledgments.} I am indebted
to Peter Doyle for his indispensable contributions.}

\begin{minipage}[t]{1\columnwidth}%
\global\long\def\manifold{M}

\global\long\def\laplace#1{\Delta_{#1}}

\global\long\def\graph{G}

\global\long\def\mate#1{#1'}

\global\long\def\other#1{\widetilde{#1}}

\global\long\def\withoutDirichlet#1{#1^{*}}

\global\long\def\lineGraphVertexColored#1{L^{\mathrm{vc}}(#1)}

\global\long\def\lineGraphEdgeColored#1{L^{\mathrm{ec}}(#1)}

\global\long\def\colorEdge{c}

\global\long\def\length{l}

\global\long\def\adjacencyMatrix#1#2{A_{#2}^{#1}}

\global\long\def\incidenceMatrix#1{B_{#1}}

\global\long\def\identityMatrix#1{I_{#1}}

\global\long\def\transplantationMatrix{T}

\global\long\def\transplantationMatrixLineGraphs{T_{L}}

\global\long\def\weight{w}

\global\long\def\colorMatrix#1{C^{#1}}

\global\long\def\nrVertices{n}

\global\long\def\nrVerticesLineGraph#1{n_{L}^{#1}}

\global\long\def\nrEdges{m}

\global\long\def\nrColors{k}

\global\long\def\vertex{v}

\global\long\def\edge{e}

\global\long\def\trace{\mathrm{tr}}

\global\long\def\setOfColoredEdges#1{E^{#1}}
\end{minipage}

\vspace{-15mm}

\section{Introduction}

Inverse spectral geometry studies the extend to which a geometric
object, e.g., a Euclidean domain, is determined by the spectral data
of an associated operator, e.g., the eigenvalues of the Laplace operator
with suitable boundary conditions. This objective is beautifully summarized
by Kac's influential question ``Can one hear the shape of a drum?''~\cite{Kac1966}.
Recently, the author~\cite{Herbrich2014} studied broken drums each
of which is modeled as a compact flat manifold $M$ with boundary
$\partial M=\overline{\partial_{D}M\cup\partial_{N}M}$, where $\partial_{D}M$
and $\partial_{N}M$ represent the attached and unattached parts of
the drumhead, respectively. The audible frequencies of such a broken
drum are determined by the eigenvalues of the Laplace-Beltrami operator~$\laplace{\manifold}$
of $\manifold$ with Dirichlet and Neumann boundary conditions along
$\partial_{D}M$ and $\partial_{N}M$, respectively. Provided that
$\partial M$ is sufficiently smooth, this operator has discrete spectrum
given by an unbounded non-decreasing sequence of non-negative eigenvalues.

Using number-theoretic ideas, Sunada~\cite{Sunada1985} developed
a celebrated method involving group actions to construct isospectral
manifolds, i.e., manifolds whose spectra coincide. It ultimately allowed
Gordon~et~al.~\cite{GordonWebbWolpert1992} to answer Kac's question
in the negative. Buser~\cite{Buser1986} distilled the combinatorial
core of Sunada's method into the transplantation method, which involves
tiled manifolds that are composed of identical building blocks, e.g.,
$\manifold$ and $\mate{\manifold}$ in Figure~\ref{fig:Tiled_manifolds}.
If $\varphi$ is an eigenfunction on $\manifold$, then its restrictions~$(\varphi_{i})_{i=1}^{4}$
to the blocks of $\manifold$ are superposed linearly as $(\sum_{i=1}^{4}\transplantationMatrix_{ij}\varphi_{i}){}_{j=1}^{4}$
on the blocks of $\mate{\manifold}$ such that the result is an eigenfunction
on $\mate{\manifold}$, and vice versa. All known pairs of isospectral
planar domains with Dirichlet boundary conditions arise in this way~\cite{BuserConwayDoyleSemmler1994},
i.e., they are transplantable.

Following~\cite{Herbrich2014}, we encode each tiled manifold with
mixed Dirichlet-Neumann boundary conditions by an edge-colored graph
with signed loops that encode boundary conditions, e.g., $\graph$
and $\mate{\graph}$ in Figure~\ref{fig:Loop-signed-graphs}. By
definition, every vertex of an edge-colored loop-signed graph~$\graph$
has one incident edge of each color, either as a link to another vertex
or as a signed loop. If~$\graph$ has $\nrColors$ edge colors, then
it is determined by its $\nrColors$-tuple of adjacency matrices $(\adjacencyMatrix{\colorEdge}{\graph})_{\colorEdge=1}^{\nrColors}$,
which are diagonally-signed permutation matrices with $\adjacencyMatrix{\colorEdge}{\graph}=(\adjacencyMatrix{\colorEdge}{\graph})^{T}=(\adjacencyMatrix{\colorEdge}{\graph})^{-1}$
given by
\[
[\adjacencyMatrix{\colorEdge}{\graph}]_{\vertex\other{\vertex}}=\begin{cases}
1 & \mbox{if }\vertex\neq\other{\vertex}\text{ and there is a }\colorEdge\mbox{-colored link between vertices }\vertex\text{ and }\other{\vertex},\\
\pm1 & \mbox{if }\vertex=\other{\vertex}\text{ and there is a }\colorEdge\mbox{-colored }N\text{- or }D\text{-loop}\text{ at vertex }\vertex,\text{ respectively,}\\
0 & \mbox{otherwise}.
\end{cases}
\]

\begin{defn}
Let $\graph$ and $\mate{\graph}$ be edge-colored or vertex-colored
graphs given by $\nrColors$-tuples of $\nrVertices\times\nrVertices$
adjacency matrices $(\adjacencyMatrix{\colorEdge}{\graph})_{\colorEdge=1}^{\nrColors}$
and $(\adjacencyMatrix{\colorEdge}{\mate{\graph}})_{\colorEdge=1}^{\nrColors}$,
respectively. Then $\graph$ and $\mate{\graph}$ are said to be
\begin{enumerate}
\item transplantable if there exists an invertible transplantation matrix
$\transplantationMatrix\in\mathbb{R}^{\nrVertices\times\nrVertices}$
such that
\[
\adjacencyMatrix{\colorEdge}{\graph}=\transplantationMatrix\adjacencyMatrix{\colorEdge}{\mate{\graph}}\transplantationMatrix^{-1}\quad\text{for every }\colorEdge\in\{1,2,\ldots,\nrColors\},
\]

\item cycle equivalent if for every finite sequence of colors $\colorEdge_{1},\colorEdge_{2},\ldots,\colorEdge_{\length}\in\{1,2,\ldots,\nrColors\}$
\[
\trace(\adjacencyMatrix{\colorEdge_{1}}{\graph}\adjacencyMatrix{\colorEdge_{2}}{\graph}\cdots\adjacencyMatrix{\colorEdge_{\length}}{\graph})=\trace(\adjacencyMatrix{\colorEdge_{1}}{\mate{\graph}}\adjacencyMatrix{\colorEdge_{2}}{\mate{\graph}}\cdots\adjacencyMatrix{\colorEdge_{\length}}{\mate{\graph}}).
\]

\end{enumerate}
Note that transplantable graphs are cycle equivalent. The following
characterization of transplantable tiled manifolds says that the converse
holds for edge-colored loop-signed graphs.\end{defn}
\begin{thm}
\label{thm:Tranplantable_manifolds}\cite{Herbrich2014} Let $\graph$
and $\mate{\graph}$ be edge-colored loop-signed graphs with the same
numbers of vertices and colors. Let $M$ and $\mate M$ be tiled manifolds
with mixed Dirichlet-Neumann boundary conditions obtained by choosing
a building block. Then the following are equivalent:
\begin{enumerate}
\item $M$ and $\mate M$ are transplantable (and therefore isospectral),\label{enu:Transplantable_tiled_manifolds}
\item $\graph$ and $\mate{\graph}$ are transplantable,\label{enu:Transplantable_edge_colored_loop_signed_graphs}
\item $\graph$ and $\mate{\graph}$ are cycle equivalent.\label{enu:Cycle_equivalent_edge_colored_loop_signed_graphs}
\end{enumerate}
\end{thm}
The equivalence of~(\ref{enu:Transplantable_tiled_manifolds}) and~(\ref{enu:Transplantable_edge_colored_loop_signed_graphs})
is shown using regularity and continuation theorems for elliptic operators,
and the equivalence of~(\ref{enu:Transplantable_edge_colored_loop_signed_graphs})
and~(\ref{enu:Cycle_equivalent_edge_colored_loop_signed_graphs})
is shown using representation theory.  In the following, we derive
further characterizations of transplantability. As is well-known,
Brownian motion on a manifold $\manifold$ has $\frac{1}{2}\laplace{\manifold}$
as its infinitesimal generator, rendering it a natural object of study
for spectral questions. Consider a particle moving on the tiled manifold
$\manifold$ in Figure~\ref{fig:Tiled_manifolds}. Each time the
particle hits $\partial_{N}M$, it is reflected back, whereas contact
with $\partial_{D}M$ destroys the particle. Since the $4$ triangular
building blocks of $\manifold$ are isometric, we only keep track
of the triangle sides visited, which corresponds to a walk on the
colored vertices of the associated directed line graph $\lineGraphVertexColored{\graph}$
shown in Figure~\ref{fig:Vertex_colored_line_graphs}. Each $N$-loop
of $\graph$ in Figure~\ref{fig:Loop-signed-graphs} contributes
$3$ directed edges of equal weight to $\lineGraphVertexColored{\graph}$
in Figure~\ref{fig:Vertex_colored_line_graphs}, whereas $D$-loops
of $\graph$ do not contribute at all. The edge-colored directed line
graph $\lineGraphEdgeColored{\graph}$ in Figure~\ref{fig:Edge_coloured_line_graphs}
is obtained by coloring edges instead of vertices. In~Section~\ref{sec:Proof},
we define $\lineGraphVertexColored{\graph}$ and $\lineGraphEdgeColored{\graph}$
rigorously, and prove the following main theorem.
\begin{thm}
\label{thm:Main_theorem}Let $\graph$ and $\mate{\graph}$ be edge-colored
loop-signed graphs with $\nrVertices$ vertices and $\nrColors$ colors.
If $\trace(\adjacencyMatrix{\colorEdge}{\graph})=\trace(\adjacencyMatrix{\colorEdge}{\mate{\graph}})$
for every $\colorEdge\in\{1,2,\ldots,\nrColors\}$, then the following
are equivalent:
\begin{enumerate}
\item $\graph$ and $\mate{\graph}$ are transplantable,\label{enu:SG_transplantable}
\item $\graph$ and $\mate{\graph}$ are cycle equivalent,\label{enu:SG_cycle_equivalent}
\item $\lineGraphVertexColored{\graph}$ and $\lineGraphVertexColored{\mate{\graph}}$
are transplantable,\label{enu:vc_LG_transplantable}
\item $\lineGraphVertexColored{\graph}$ and $\lineGraphVertexColored{\mate{\graph}}$
are cycle equivalent,\label{enu:vc_LG_cycle_equivalent}
\item $\lineGraphEdgeColored{\graph}$ and $\lineGraphEdgeColored{\mate{\graph}}$
are transplantable,\label{enu:ec_LG_transplantable}
\item $\lineGraphEdgeColored{\graph}$ and $\lineGraphEdgeColored{\mate{\graph}}$
are cycle equivalent.\label{enu:ec_LG_cycle_equivalent}
\end{enumerate}
If any of the above conditions holds, then there exists an invertible
transplantation matrix for both~(\ref{enu:vc_LG_transplantable})
and~(\ref{enu:ec_LG_transplantable}) that is the direct sum of square
matrices of sizes $(\trace(\identityMatrix{\nrVertices}+\adjacencyMatrix{\colorEdge}{\graph})/2)_{\colorEdge=1}^{\nrColors}$.
\end{thm}
Theorem~\ref{thm:Main_theorem} has the following representation-theoretic
interpretation. For $K=\{1,2,\ldots,\nrColors\}$, we denote the free
group on $K$ by $F(K)$, and the free semigroup on $K$ by $K^{+}$.
The graphs $\graph$ and $\mate{\graph}$ give rise to a pair of representations
of $F(K)$ by virtue of $\colorEdge^{\pm1}\mapsto\adjacencyMatrix{\colorEdge}{\graph}$
and $\colorEdge^{\pm1}\mapsto\adjacencyMatrix{\colorEdge}{\mate{\graph}}$,
respectively. Similarly, $\lineGraphVertexColored{\graph}$ and $\lineGraphVertexColored{\mate{\graph}}$,
as well as $\lineGraphEdgeColored{\graph}$ and $\lineGraphEdgeColored{\mate{\graph}}$,
give rise to pairs of representations of $K^{+}$. If the assumptions
of Theorem~\ref{thm:Main_theorem} are satisfied, and the representations
of one pair are equivalent or have equal characters, then both are
true for all pairs.

It is worth mentioning that~\cite{Herbrich2014} gives examples of
non-isomorphic graphs $\graph$ and $\mate{\graph}$ as in Theorem~\ref{thm:Main_theorem}
that have $4$ vertices, no $N$-loops, and isomorphic line graphs
$\lineGraphVertexColored{\graph}$ and~$\lineGraphVertexColored{\mate{\graph}}$.
These pairs closely resemble the single exception of the classical
Whitney graph isomorphism theorem~\cite{Whitney1932} which states
that two uncolored connected graphs without loops or parallel edges
are isomorphic if and only if their line graphs are isomorphic, with
the exception of the triangle graph $K_{3}$ and the star graph $S_{3}=K_{1,3}$,
which both have $K_{3}$ as their line graph.

\begin{figure}
\subfloat[Tiled manifolds\label{fig:Tiled_manifolds}]{\begin{centering}
\psfrag{1}{$1$} 
\psfrag{2}{$2$} 
\psfrag{3}{$3$} 
\psfrag{4}{$4$} %
\begin{tabular}{>{\centering}p{32mm}c}
\centering{}\hspace{-2mm}\raisebox{19mm}{$\manifold$}\hspace{-2mm}\includegraphics[scale=0.45]{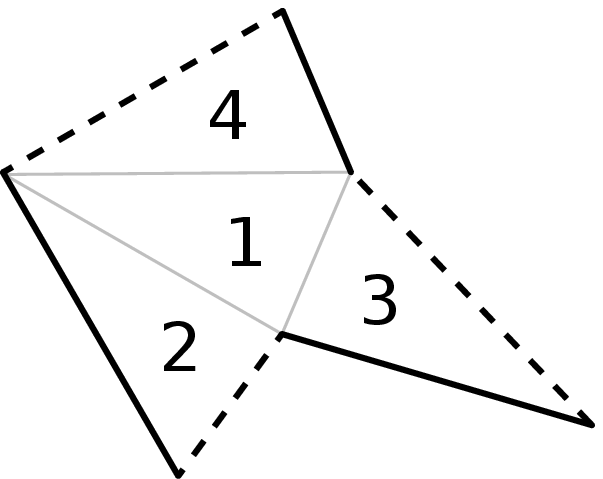} & \hspace{-2mm}\raisebox{19mm}{$\mate{\manifold}$}\hspace{-3mm}\includegraphics[scale=0.45]{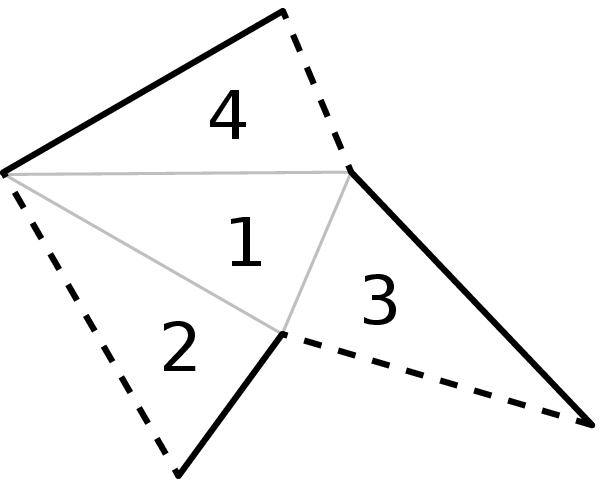}\tabularnewline
\end{tabular}
\par\end{centering}

}\hfill{}\subfloat[Triangular building block\label{fig:Building_block}]{\begin{centering}
\qquad{}\psfrag{s}{$s$} 
\psfrag{w}{$w$} 
\psfrag{z}{$z$} %
\begin{tabular}{>{\centering}p{40mm}}
\centering{}%
\begin{tabular}{c}
\includegraphics[scale=0.45]{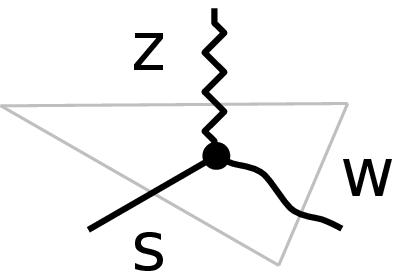}\tabularnewline
\tabularnewline
\end{tabular}\tabularnewline
\end{tabular}\qquad{}
\par\end{centering}

}\hfill{}

\subfloat[Edge-colored loop-signed graphs\label{fig:Loop-signed-graphs}]{\begin{centering}
\enskip{}\psfrag{D}{\hspace{-0.25mm}\raisebox{0.18mm}{D}}
\psfrag{N}{\hspace{-0.25mm}\raisebox{0.18mm}{N}} %
\begin{tabular}{>{\centering}p{30mm}c}
\centering{}\hspace{-2mm}\raisebox{18mm}{$\graph$}\hspace{-1mm}\includegraphics[scale=0.55]{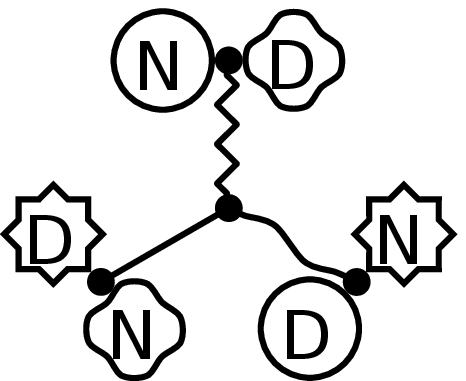} & \hspace{-2mm}\raisebox{18mm}{$\mate{\graph}$}\hspace{-2mm}\includegraphics[scale=0.55]{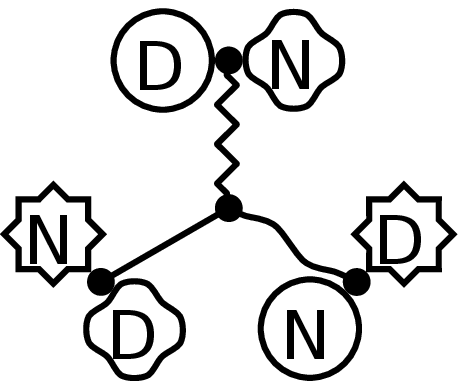}\tabularnewline
\end{tabular}\enskip{}
\par\end{centering}

}\hfill{}\subfloat[Adjacency and transplantation matrix\label{fig:SG_adjacency_and_transplantation_matrix}]{\begin{centering}
\quad{}\setlength{\arraycolsep}{1pt}%
\begin{tabular}{cc}
$s\adjacencyMatrix s{\graph}+w\adjacencyMatrix w{\graph}+z\adjacencyMatrix z{\graph}$ & $T$ ($T^{T}=3T^{-1}$)\tabularnewline
$\left(\begin{array}{cccc}
0 & s & w & z\\
s & w-z & 0 & 0\\
w & 0 & z-s & 0\\
z & 0 & 0 & s-w
\end{array}\right)$ & \setlength{\arraycolsep}{2pt}$\left(\begin{array}{cccc}
0 & 1 & 1 & 1\\
1 & 0 & 1 & -1\\
1 & -1 & 0 & 1\\
1 & 1 & -1 & 0
\end{array}\right)$\tabularnewline
\end{tabular}\setlength{\arraycolsep}{\myArraycolsep}
\par\end{centering}

}

\subfloat[Vertex-colored directed line graphs\label{fig:Vertex_colored_line_graphs}]{\begin{centering}
\begin{tabular}{>{\centering}p{32mm}c}
\centering{}\hspace{-2mm}\raisebox{24mm}{$\lineGraphVertexColored{\graph}$}\hspace{-6mm}\psfrag{a}{} 
\psfrag{A}{} 
\psfrag{b}{} 
\psfrag{B}{} 
\psfrag{c}{} 
\psfrag{C}{} 
\psfrag{d}{} 
\psfrag{D}{} 
\psfrag{1}{\raisebox{-0.5mm}{$1s$}} 
\psfrag{2}{\hspace{-2mm}$2s$} 
\psfrag{3}{\raisebox{-0.5mm}{\hspace{-1.5mm}$3w$}} 
\psfrag{4}{$4w$} 
\psfrag{5}{\hspace{-2mm}$5z$} 
\psfrag{6}{\hspace{-0.5mm}$6z$} \includegraphics[scale=0.45]{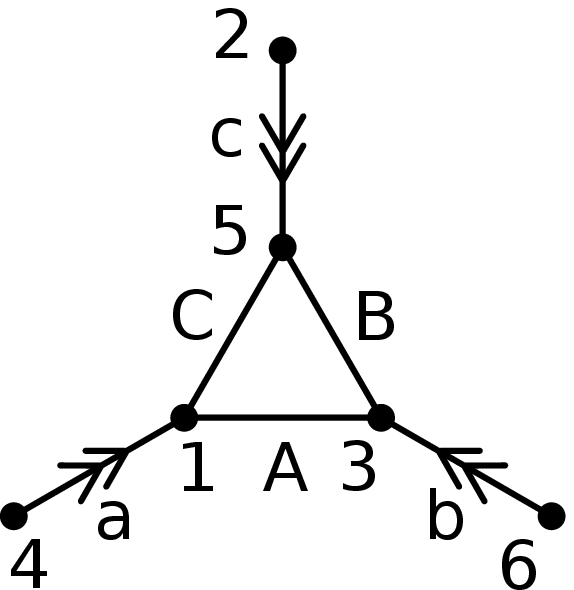} & \hspace{-5mm}\raisebox{24mm}{$\lineGraphVertexColored{\mate{\graph}}$}\hspace{-5mm}\psfrag{a}{} 
\psfrag{A}{} 
\psfrag{b}{} 
\psfrag{B}{} 
\psfrag{c}{} 
\psfrag{C}{} 
\psfrag{d}{} 
\psfrag{D}{} 
\psfrag{1}{\raisebox{-0.5mm}{$1s$}} 
\psfrag{2}{\hspace{-0.5mm}$2s$} 
\psfrag{3}{\raisebox{-0.5mm}{\hspace{-1.5mm}$3w$}} 
\psfrag{4}{\hspace{-2.5mm}$4w$} 
\psfrag{5}{\hspace{-2mm}$5z$} 
\psfrag{6}{$6z$} \includegraphics[scale=0.45]{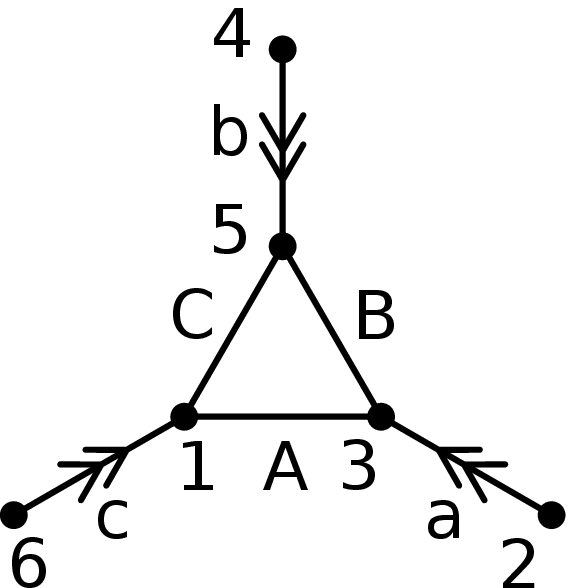}\tabularnewline
\end{tabular}\enspace{}
\par\end{centering}

}\hfill{}\subfloat[Adjacency and transplantation matrix\label{fig:LG_vc_adjacency_and_transplantation_matrix}]{\begin{centering}
\begin{tabular}[b]{cc}
\setlength{\arraycolsep}{2pt}$\left(\begin{array}{cccccc}
0 & 0 & s & s & s & 0\\
0 & 0 & 0 & 0 & 2s & 0\\
w & 0 & 0 & 0 & w & w\\
2w & 0 & 0 & 0 & 0 & 0\\
z & z & z & 0 & 0 & 0\\
0 & 0 & 2z & 0 & 0 & 0
\end{array}\right)$ & \setlength{\arraycolsep}{2pt}$\left(\begin{array}{cccccc}
1 & 1 & 0 & 0 & 0 & 0\\
2 & -1 & 0 & 0 & 0 & 0\\
0 & 0 & 1 & 1 & 0 & 0\\
0 & 0 & 2 & -1 & 0 & 0\\
0 & 0 & 0 & 0 & 1 & 1\\
0 & 0 & 0 & 0 & 2 & -1
\end{array}\right)$\setlength{\arraycolsep}{\myArraycolsep}\tabularnewline
\end{tabular}
\par\end{centering}

}

\subfloat[Edge-colored directed line graphs\label{fig:Edge_coloured_line_graphs}]{\begin{centering}
\psfrag{a}{$a$} 
\psfrag{A}{\raisebox{-0.5mm}{$a$}} 
\psfrag{b}{$b$} 
\psfrag{B}{$b$} 
\psfrag{c}{$c$} 
\psfrag{C}{$c$} 
\psfrag{d}{$d$} 
\psfrag{D}{$d$} 
\psfrag{1}{\raisebox{-0.5mm}{\hspace{0.5mm}$1$}} 
\psfrag{2}{$2$} 
\psfrag{3}{\raisebox{-0.5mm}{$3$}} 
\psfrag{4}{$4$} 
\psfrag{5}{$5$} 
\psfrag{6}{$6$} 
\psfrag{7}{$7$} 
\psfrag{8}{$8$} %
\begin{tabular}{>{\centering}p{32mm}c}
\centering{}\hspace{-2mm}\raisebox{24mm}{$\lineGraphEdgeColored{\graph}$}\hspace{-6mm}\includegraphics[scale=0.45]{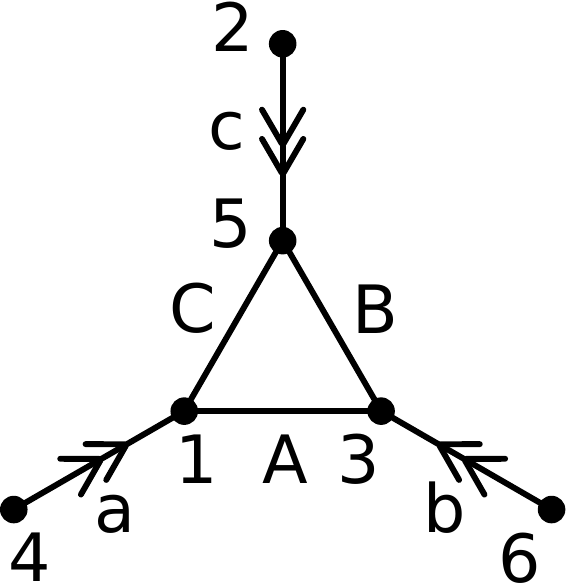} & \hspace{-5mm}\raisebox{24mm}{$\lineGraphEdgeColored{\mate{\graph}}$}\hspace{-5mm}\includegraphics[scale=0.45]{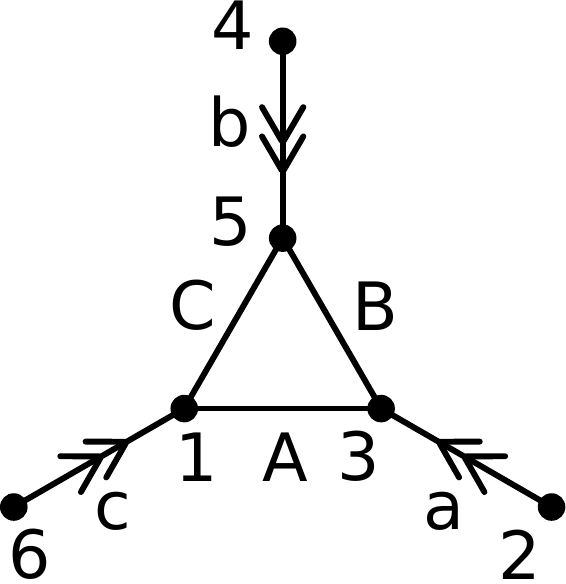}\tabularnewline
\end{tabular}
\par\end{centering}

}\hfill{}\subfloat[Adjacency and transplantation matrix\label{fig:LG_ec_adjacency_and_transplantation_matrix}]{\begin{centering}
\begin{tabular}[b]{cc}
\setlength{\arraycolsep}{2pt}$\left(\begin{array}{cccccc}
0 & 0 & a & a & c & 0\\
0 & 0 & 0 & 0 & 2c & 0\\
a & 0 & 0 & 0 & b & b\\
2a & 0 & 0 & 0 & 0 & 0\\
c & c & b & 0 & 0 & 0\\
0 & 0 & 2b & 0 & 0 & 0
\end{array}\right)$ & \setlength{\arraycolsep}{2pt}$\left(\begin{array}{cccccc}
1 & 1 & 0 & 0 & 0 & 0\\
2 & -1 & 0 & 0 & 0 & 0\\
0 & 0 & 1 & 1 & 0 & 0\\
0 & 0 & 2 & -1 & 0 & 0\\
0 & 0 & 0 & 0 & 1 & 1\\
0 & 0 & 0 & 0 & 2 & -1
\end{array}\right)$\setlength{\arraycolsep}{\myArraycolsep}\tabularnewline
\end{tabular}
\par\end{centering}

}

\caption{Graph representations of a pair of transplantable tiled manifolds.
The adjacency matrices belong to the respective first graph. The types
of line $(straight,wavy,zigzag)$ represent the edge colors $(s,w,z)$
of $\graph$ and $\mate{\graph}$. The graphs $\lineGraphEdgeColored{\graph}$
and $\lineGraphEdgeColored{\mate{\graph}}$ have edge colors $(a,b,c)=(\{s,w\},\{w,z\},\{s,z\})$.}
\end{figure}

We want to point out the results in~\cite{McDonaldMeyers2003,OrenBand2012},
which initiated our investigations. In~\cite{McDonaldMeyers2003},
McDonald and Meyers consider the finitely many known pairs of transplantable
planar domains with pure Dirichlet boundary conditions~\cite{BuserConwayDoyleSemmler1994},
and introduce their line graph construction, which, in our notation,
corresponds to the assignment $\manifold\mapsto\lineGraphEdgeColored{\graph}$.
For~each of the pairs in~\cite{BuserConwayDoyleSemmler1994}, they
verify that the associated edge-colored line graphs are cospectral
with respect to a certain discrete Laplace operator. In~\cite{OrenBand2012},
Oren and Band note that these graphs are also cospectral with respect
to their weighted adjacency matrices. However, the line graph construction
was neither known to always produce cospectral graphs, nor could it
deal with Neumann boundary conditions, and it had not been noticed
that there exist canonical transplantations as in the second part
of Theorem~\ref{thm:Main_theorem}.

\section{\label{sec:Proof}Colored directed line graphs}

\noindent Let $\graph$ be an edge-colored loop-signed graph with
$\nrVertices$ vertices and adjacency matrices $(\adjacencyMatrix{\colorEdge}{\graph})_{\colorEdge=1}^{\nrColors}$.
In particular, $\trace(\identityMatrix{\nrVertices}+\adjacencyMatrix{\colorEdge}{\graph})/2$
equals the number of $\colorEdge$-colored links and $N$-loops of
$\graph$.
\begin{defn}
Let $\withoutDirichlet{\graph}$ be the graph obtained from $\graph$
by removing all $D$-loops. Let
\[
\nrVerticesLineGraph{}=\sum_{\colorEdge=1}^{\nrColors}\frac{\trace(\identityMatrix{\nrVertices}+\adjacencyMatrix{\colorEdge}{\graph})}{2}
\]
denote the number of edges of $\withoutDirichlet{\graph}$. The vertex-colored
directed line graph $\lineGraphVertexColored{\graph}$ of $\graph$
has one $c$-colored vertex for each $\colorEdge$-colored edge of
$\withoutDirichlet{\graph}$, and two vertices of $\lineGraphVertexColored{\graph}$
are connected if and only if the corresponding edges in $\withoutDirichlet{\graph}$
are incident. More precisely, $\lineGraphVertexColored{\graph}$ is
defined by its $\nrColors$-tuple of $\nrVerticesLineGraph{}\times\nrVerticesLineGraph{}$
adjacency matrices $(\adjacencyMatrix{\colorEdge}{\lineGraphVertexColored{\graph}})_{\colorEdge=1}^{\nrColors}$
given by
\[
[\adjacencyMatrix{\colorEdge}{\lineGraphVertexColored{\graph}}]_{\edge\other{\edge}}=\begin{cases}
2 & \mbox{if edge }\edge\text{ of }\withoutDirichlet{\graph}\text{ is }\colorEdge\text{-colored and shares all of its vertices with edge }\other{\edge}\neq\edge,\\
1 & \mbox{if edge }\edge\text{ of }\withoutDirichlet{\graph}\text{ is a }\colorEdge\text{-colored link and shares one vertex with edge }\other{\edge}\neq\edge,\\
0 & \mbox{otherwise}.
\end{cases}
\]
The edge-colored directed line graph $\lineGraphEdgeColored{\graph}$
has $\nrVerticesLineGraph{}$ vertices, colors $\{\{\colorEdge,\other{\colorEdge}\}\mid1\leq\colorEdge<\other{\colorEdge}\leq\nrColors\}$,
and is obtained from $\lineGraphVertexColored{\graph}$ by coloring
its edges with the colors of their incident vertices. More precisely,
$\lineGraphEdgeColored{\graph}$ is defined by its $\binom{\nrColors}{2}$
adjacency matrices $(\adjacencyMatrix{\{\colorEdge,\other{\colorEdge}\}}{\lineGraphEdgeColored{\graph}})_{1\leq\colorEdge<\other{\colorEdge}\leq\nrColors}$
given by
\[
[\adjacencyMatrix{\{\colorEdge,\other{\colorEdge}\}}{\lineGraphEdgeColored{\graph}}]_{\edge\other{\edge}}=\begin{cases}
[\adjacencyMatrix{\colorEdge}{\lineGraphVertexColored{\graph}}]_{\edge\other{\edge}}+[\adjacencyMatrix{\other{\colorEdge}}{\lineGraphVertexColored{\graph}}]_{\edge\other{\edge}} & \text{if the set of colors of edges }\edge\text{ and }\other{\edge}\text{ of }\withoutDirichlet{\graph}\text{ is }\{\colorEdge,\other{\colorEdge}\},\\
0 & \mbox{otherwise}.
\end{cases}
\]

\end{defn}
We note that if $\graph$ has no $N$-loops or parallel links, then
$\lineGraphEdgeColored{\graph}$ is a simple edge-colored undirected
graph, meaning it has symmetric $\{0,1\}$-adjacency matrices with
zero diagonal.
\begin{defn}
\noindent Let $\incidenceMatrix{\weight}\in\{0,1,\weight\}^{n\times\nrVerticesLineGraph{}}$
be the weighted incidence matrix of $\withoutDirichlet{\graph}$ given
by 
\[
[\incidenceMatrix{\weight}]_{\vertex\edge}=\begin{cases}
\weight & \text{if edge }\edge\text{ of }\withoutDirichlet{\graph}\text{ is an }N\text{-loop incident to vertex }\vertex,\\
1 & \text{if edge }\edge\text{ of }\withoutDirichlet{\graph}\text{ is a link incident to vertex }\vertex,\\
0 & \text{otherwise.}
\end{cases}
\]
For each $\colorEdge\in\{1,2,\ldots,\nrColors\}$, let $\colorMatrix{\colorEdge}\in\{0,1\}^{\nrVerticesLineGraph{}\times\nrVerticesLineGraph{}}$
be the diagonal matrix given by
\[
[\colorMatrix{\colorEdge}]_{\edge\edge}=\begin{cases}
1 & \text{if edge }\edge\text{ of }\withoutDirichlet{\graph}\text{ has color }\colorEdge,\\
0 & \text{otherwise.}
\end{cases}
\]
\end{defn}
\begin{lem}
\label{lem:Incidence_decomposition}For every $\colorEdge\in\{1,2,\ldots,\nrColors\}$,
\[
\adjacencyMatrix{\colorEdge}{\graph}=\incidenceMatrix 1\colorMatrix{\colorEdge}\incidenceMatrix 2^{T}-\identityMatrix{\nrVertices}\qquad\text{and}\qquad\adjacencyMatrix{\colorEdge}{\lineGraphVertexColored{\graph}}=\colorMatrix{\colorEdge}(\incidenceMatrix 2^{T}\incidenceMatrix 1-2\identityMatrix{\nrVerticesLineGraph{}}).
\]
\end{lem}
\begin{proof}
We show the matrix equalities row by row. If $M$ is a matrix, let
$[M]_{m}$ denote its row~$m$. Let $\vertex\in\{1,2,\ldots,\nrVertices\}$.
Recall that vertex $\vertex$ of $\graph$ has exactly one $\colorEdge$-colored
incident edge $\edge$, either as a link, $N$-loop, or $D$-loop,
respectively. If $\edge$ is a link to vertex $\other{\vertex}\neq\vertex$,
then $[\adjacencyMatrix{\colorEdge}{\graph}]_{\vertex\other{\vertex}}=1$
and $[\incidenceMatrix 1\colorMatrix{\colorEdge}]_{\vertex\edge}=1$
are the only non-vanishing entries in $[\adjacencyMatrix{\colorEdge}{\graph}]_{\vertex}$
and $[\incidenceMatrix 1\colorMatrix{\colorEdge}]_{\vertex}$, respectively.
In particular, $[\incidenceMatrix 1\colorMatrix{\colorEdge}\incidenceMatrix 2^{T}]_{\vertex}$
equals $[\incidenceMatrix 2^{T}]_{\edge}$ whose non-zero entries
are $[\incidenceMatrix 2^{T}]_{\edge\vertex}=[\incidenceMatrix 2^{T}]_{\edge\other{\vertex}}=1$.
Similarly, if $\edge$ is a $D$-link, then $[\adjacencyMatrix{\colorEdge}{\graph}]_{\vertex\vertex}=-1$
and $[\incidenceMatrix 1\colorMatrix{\colorEdge}\incidenceMatrix 2^{T}]_{\vertex}=0$.
Finally, if $\edge$ is an $N$-link, then $[\adjacencyMatrix{\colorEdge}{\graph}]_{\vertex\vertex}=1$
and $[\incidenceMatrix 1\colorMatrix{\colorEdge}]_{\vertex\edge}=1$
are the only non-vanishing entries in their respective rows. In particular,
the same is true for $[\incidenceMatrix 1\colorMatrix{\colorEdge}\incidenceMatrix 2^{T}]_{\vertex\vertex}=[\incidenceMatrix 2^{T}]_{\edge\vertex}=2$,
which shows the first equality.

Let $\edge\in\{1,2,\ldots,\nrVerticesLineGraph{}\}$. If edge $\edge$
of $\withoutDirichlet{\graph}$ is not $\colorEdge$-colored, then
$[\adjacencyMatrix{\colorEdge}{\lineGraphVertexColored{\graph}}]_{\edge}=0$
and $[\colorMatrix{\colorEdge}]_{\edge}=0$. We therefore assume that
$\edge$ is $\colorEdge$-colored, in which case $[\colorMatrix{\colorEdge}(\incidenceMatrix 2^{T}\incidenceMatrix 1-2\identityMatrix{\nrVerticesLineGraph{}})]_{\edge}=[\incidenceMatrix 2^{T}\incidenceMatrix 1-2\identityMatrix{\nrVerticesLineGraph{}}]_{\edge}$.
If $\edge$ is a link between vertices $\vertex$ and $\other{\vertex}$
of $\graph$, then $[\incidenceMatrix 2^{T}]_{\edge\vertex}=[\incidenceMatrix 2^{T}]_{\edge\other{\vertex}}=1$
are the only non-zero entries in $[\incidenceMatrix 2^{T}]_{\edge}$.
In particular, $[\incidenceMatrix 2^{T}\incidenceMatrix 1]_{\edge}=[\incidenceMatrix 1]_{\vertex}+[\incidenceMatrix 1]_{\other{\vertex}}$,
which coincides with $[\adjacencyMatrix{\colorEdge}{\lineGraphVertexColored{\graph}}+2\identityMatrix{\nrVerticesLineGraph{}}]_{\edge}$.
Similarly, if $\edge$ is an $N$-loop at $\vertex$, then $[\incidenceMatrix 2^{T}]_{\edge\vertex}=2$
is the only non-zero entry in $[\incidenceMatrix 2^{T}]_{\edge}$
which gives $[\incidenceMatrix 2^{T}\incidenceMatrix 1]_{\edge}=2[\incidenceMatrix 1]_{\vertex}=[\adjacencyMatrix{\colorEdge}{\lineGraphVertexColored{\graph}}+2\identityMatrix{\nrVerticesLineGraph{}}]_{\edge}$.\end{proof}
\begin{lem}
\label{lem:Cycle_equivalences} In Theorem~\ref{thm:Main_theorem},
the statements~(\ref{enu:SG_cycle_equivalent}), (\ref{enu:vc_LG_cycle_equivalent}),
and (\ref{enu:ec_LG_cycle_equivalent}) are equivalent.\end{lem}
\begin{proof}
We start by showing that~(\ref{enu:SG_cycle_equivalent}) and~(\ref{enu:vc_LG_cycle_equivalent})
are equivalent, which amounts to showing that the traces of products
of adjacency matrices of $\graph$ determine those of $\lineGraphVertexColored{\graph}$,
and vice versa. Since $\trace(\identityMatrix{\nrVertices})=\nrVertices$
and $(\trace(\adjacencyMatrix{\colorEdge}{\graph}))_{\colorEdge=1}^{\nrColors}$
are given by assumption, $\trace(\identityMatrix{\nrVerticesLineGraph{}})=\sum_{\colorEdge=1}^{\nrColors}(\trace(\identityMatrix{\nrVertices}+\adjacencyMatrix{\colorEdge}{\graph}))/2$
can be assumed as given as well. For $\colorEdge\in\{1,2,\ldots,\nrColors\}$,
we have $\trace(\adjacencyMatrix{\colorEdge}{\lineGraphVertexColored{\graph}})=0$,
$(\adjacencyMatrix{\colorEdge}{\graph})^{2}=\identityMatrix{\nrVertices}$,
and $(\adjacencyMatrix{\colorEdge}{\lineGraphVertexColored{\graph}})^{2}=0$.
It therefore suffices to consider products of adjacency matrices with
cyclically square-free color sequences, i.e., sequences of the form
$\colorEdge_{1},\colorEdge_{2}\ldots,\colorEdge_{\length}\in\{1,2,\ldots,\nrColors\}$
with $\colorEdge_{1}\neq\colorEdge_{\length}$ and $\colorEdge_{i}\neq\colorEdge_{i+1}$
for $i\in\{1,2,\ldots,\length-1\}$. As $\colorMatrix{\colorEdge_{1}}\colorMatrix{\colorEdge_{2}}=\colorMatrix{\colorEdge_{2}}\colorMatrix{\colorEdge_{1}}=0$,
Lemma~\ref{lem:Incidence_decomposition} yields
\[
\adjacencyMatrix{\colorEdge_{1}}{\lineGraphVertexColored{\graph}}\adjacencyMatrix{\colorEdge_{2}}{\lineGraphVertexColored{\graph}}=(\colorMatrix{\colorEdge_{1}}\incidenceMatrix 2^{T}\incidenceMatrix 1-2\colorMatrix{\colorEdge_{1}})(\colorMatrix{\colorEdge_{2}}\incidenceMatrix 2^{T}\incidenceMatrix 1-2\colorMatrix{\colorEdge_{2}})=\colorMatrix{\colorEdge_{1}}\incidenceMatrix 2^{T}\incidenceMatrix 1\colorMatrix{\colorEdge_{2}}\incidenceMatrix 2^{T}\incidenceMatrix 1-2\colorMatrix{\colorEdge_{1}}\incidenceMatrix 2^{T}\incidenceMatrix 1\colorMatrix{\colorEdge_{2}},
\]
which has trace
\[
\trace(\adjacencyMatrix{\colorEdge_{1}}{\lineGraphVertexColored{\graph}}\adjacencyMatrix{\colorEdge_{2}}{\lineGraphVertexColored{\graph}})=\trace(\incidenceMatrix 1\colorMatrix{\colorEdge_{1}}\incidenceMatrix 2^{T}\incidenceMatrix 1\colorMatrix{\colorEdge_{2}}\incidenceMatrix 2^{T})-2\,\trace(\colorMatrix{\colorEdge_{2}}\colorMatrix{\colorEdge_{1}}\incidenceMatrix 2^{T}\incidenceMatrix 1)=\trace((\adjacencyMatrix{\colorEdge_{1}}{\graph}+\identityMatrix{\nrVertices})(\adjacencyMatrix{\colorEdge_{2}}{\graph}+\identityMatrix{\nrVertices})).
\]
Similarly,
\[
\trace(\adjacencyMatrix{\colorEdge_{1}}{\lineGraphVertexColored{\graph}}\adjacencyMatrix{\colorEdge_{2}}{\lineGraphVertexColored{\graph}}\cdots\adjacencyMatrix{\colorEdge_{\length}}{\lineGraphVertexColored{\graph}})=\trace((\adjacencyMatrix{\colorEdge_{1}}{\graph}+\identityMatrix{\nrVertices})(\adjacencyMatrix{\colorEdge_{2}}{\graph}+\identityMatrix{\nrVertices})\cdots(\adjacencyMatrix{\colorEdge_{\length}}{\graph}+\identityMatrix{\nrVertices})),
\]
which gives the desired statement by induction on $\length$.

We finish by showing that~(\ref{enu:vc_LG_cycle_equivalent}) and~(\ref{enu:ec_LG_cycle_equivalent})
are equivalent, which is essentially due to the fact that $\lineGraphVertexColored{\graph}$
and $\lineGraphEdgeColored{\graph}$ have the same set of cycles,
i.e., closed walks on their vertices. Proceeding as above, we note
that $\trace(\adjacencyMatrix{\{\colorEdge,\other{\colorEdge}\}}{\lineGraphEdgeColored{\graph}})=0$
for every $\colorEdge,\other{\colorEdge}\in\{1,2,\ldots,\nrColors\}$
with $\colorEdge\neq\other{\colorEdge}$. Also, $\adjacencyMatrix{\{\colorEdge_{1},\other{\colorEdge}_{1}\}}{\lineGraphEdgeColored{\graph}}\adjacencyMatrix{\{\colorEdge_{2},\other{\colorEdge}_{2}\}}{\lineGraphEdgeColored{\graph}}=0$
whenever $\{\colorEdge_{1},\other{\colorEdge}_{1}\}\cap\{\colorEdge_{2},\other{\colorEdge}_{2}\}=\varnothing$.
Thus, 
\[
\trace(\adjacencyMatrix{\{\colorEdge_{1},\other{\colorEdge}_{1}\}}{\lineGraphEdgeColored{\graph}}\adjacencyMatrix{\{\colorEdge_{2},\other{\colorEdge}_{2}\}}{\lineGraphEdgeColored{\graph}}\cdots\adjacencyMatrix{\{\colorEdge_{\length},\other{\colorEdge}_{\length}\}}{\lineGraphEdgeColored{\graph}})=0
\]
 unless $\{\colorEdge_{1},\other{\colorEdge}_{1}\}\cap\{\colorEdge_{\length},\other{\colorEdge}_{\length}\}\neq\varnothing$
and $\{\colorEdge_{i},\other{\colorEdge}_{i}\}\cap\{\colorEdge_{i+1},\other{\colorEdge}_{i+1}\}\neq\varnothing$
for $i\in\{1,2,\ldots,\length-1\}$. Due to the cyclic invariance
of the trace, every possibly non-zero trace is of the form
\[
\trace\big((\adjacencyMatrix{\{\colorEdge,\other{\colorEdge}\}}{\lineGraphEdgeColored{\graph}})^{j}\big)=2\,\trace\big((\adjacencyMatrix{\colorEdge}{\lineGraphVertexColored{\graph}}\adjacencyMatrix{\other{\colorEdge}}{\lineGraphVertexColored{\graph}})^{j}\big),\qquad\text{where }1\leq\colorEdge<\other{\colorEdge}\leq\nrColors\text{ and }j\geq1,
\]
or
\[
\trace(\adjacencyMatrix{\{\colorEdge_{1},\colorEdge_{2}\}}{\lineGraphEdgeColored{\graph}}\adjacencyMatrix{\{\colorEdge_{2},\colorEdge_{3}\}}{\lineGraphEdgeColored{\graph}}\cdots\adjacencyMatrix{\{\colorEdge_{\length},\colorEdge_{1}\}}{\lineGraphEdgeColored{\graph}})=\trace(\adjacencyMatrix{\colorEdge_{1}}{\lineGraphVertexColored{\graph}}\adjacencyMatrix{\colorEdge_{2}}{\lineGraphVertexColored{\graph}}\cdots\adjacencyMatrix{\colorEdge_{\length}}{\lineGraphVertexColored{\graph}}),\qquad\text{where }\colorEdge_{1}\neq\colorEdge_{3}.
\]
The first equation follows from the observation that $[(\adjacencyMatrix{\colorEdge}{\lineGraphVertexColored{\graph}}\adjacencyMatrix{\other{\colorEdge}}{\lineGraphVertexColored{\graph}})^{j}]_{\edge\edge}\neq0$
only if vertex $\edge$ of $\lineGraphVertexColored{\graph}$ has
color $\colorEdge$, in which case it equals the number of closed
walks that start at $\edge$ and have color sequence $\colorEdge,\other{\colorEdge},\colorEdge\ldots,\other{\colorEdge}$,
likewise for $[(\adjacencyMatrix{\other{\colorEdge}}{\lineGraphVertexColored{\graph}}\adjacencyMatrix{\colorEdge}{\lineGraphVertexColored{\graph}})^{j}]_{\edge\edge}$.
Thus,
\[
[(\adjacencyMatrix{\{\colorEdge,\other{\colorEdge}\}}{\lineGraphEdgeColored{\graph}})^{j}]_{\edge\edge}=[(\adjacencyMatrix{\colorEdge}{\lineGraphVertexColored{\graph}}\adjacencyMatrix{\other{\colorEdge}}{\lineGraphVertexColored{\graph}})^{j}]_{\edge\edge}+[(\adjacencyMatrix{\other{\colorEdge}}{\lineGraphVertexColored{\graph}}\adjacencyMatrix{\colorEdge}{\lineGraphVertexColored{\graph}})^{j}]_{\edge\edge}.
\]
The second equation follows in a similar fashion since $[\adjacencyMatrix{\{\colorEdge_{1},\colorEdge_{2}\}}{\lineGraphEdgeColored{\graph}}\adjacencyMatrix{\{\colorEdge_{2},\colorEdge_{3}\}}{\lineGraphEdgeColored{\graph}}\cdots\adjacencyMatrix{\{\colorEdge_{\length},\colorEdge_{1}\}}{\lineGraphEdgeColored{\graph}}]_{\edge}\neq0$
only if vertex $\edge$ of $\lineGraphVertexColored{\graph}$ has
color $\colorEdge_{1}$.
\end{proof}
Theorem~\ref{thm:Tranplantable_manifolds} and Lemma~\ref{lem:Cycle_equivalences}
show that~(\ref{enu:SG_transplantable}),~(\ref{enu:SG_cycle_equivalent}),~(\ref{enu:vc_LG_cycle_equivalent}),
and~(\ref{enu:ec_LG_cycle_equivalent}) in Theorem~\ref{thm:Main_theorem}
are equivalent. Since transplantability of graphs implies their cycle
equivalence, e.g., (\ref{enu:vc_LG_transplantable})$\Rightarrow$(\ref{enu:vc_LG_cycle_equivalent})
and (\ref{enu:ec_LG_transplantable})$\Rightarrow$(\ref{enu:ec_LG_cycle_equivalent}),
Theorem~\ref{thm:Main_theorem} will be proven once we have shown
that~(\ref{enu:SG_transplantable}) implies the existence of a block
diagonal transplantation matrix for both~(\ref{enu:vc_LG_transplantable})
and~(\ref{enu:ec_LG_transplantable}) as claimed, i.e., (\ref{enu:SG_transplantable})$\Rightarrow$(\ref{enu:vc_LG_transplantable})
and (\ref{enu:SG_transplantable})$\Rightarrow$(\ref{enu:ec_LG_transplantable}).

In the following, let $\graph$ and $\mate{\graph}$ be transplantable
edge-colored loop-signed graphs with $\nrVertices\times\nrVertices$
adjacency matrices $(\adjacencyMatrix{\colorEdge}{\graph})_{\colorEdge=1}^{\nrColors}$
and $(\adjacencyMatrix{\colorEdge}{\mate{\graph}})_{\colorEdge=1}^{\nrColors}$,
respectively. Let $\transplantationMatrix\in\mathbb{R}^{\nrVertices\times\nrVertices}$
be an invertible transplantation matrix satisfying $\adjacencyMatrix{\colorEdge}{\graph}\transplantationMatrix=\transplantationMatrix\adjacencyMatrix{\colorEdge}{\mate{\graph}}$
for every $\colorEdge\in\{1,2,\ldots,\nrColors\}$. In particular,
\[
\nrVerticesLineGraph{\colorEdge}=\frac{\trace(\identityMatrix{\nrVertices}+\adjacencyMatrix{\colorEdge}{\graph})}{2}=\frac{\trace(\identityMatrix{\nrVertices}+\adjacencyMatrix{\colorEdge}{\mate{\graph}})}{2},
\]
which equals the number of $\colorEdge$-colored links and $N$-loops
of $\graph$ or $\mate{\graph}$, respectively. Each of the graphs
$\lineGraphVertexColored{\graph}$ and $\lineGraphVertexColored{\mate{\graph}}$
has $\nrVerticesLineGraph{}=\nrVerticesLineGraph 1+\nrVerticesLineGraph 2+\ldots+\nrVerticesLineGraph{\nrColors}$
vertices, which we number accordingly, i.e., the respective first
$\nrVerticesLineGraph 1$ vertices have color $1$, followed by $\nrVerticesLineGraph 2$
vertices of color~$2$, and so on. Let $\edge,\mate{\edge}\in\{1,2,\ldots,\nrVerticesLineGraph{}\}$.
We denote the color of edge $\edge$ of $\withoutDirichlet{\graph}$
by $\colorEdge$, and its incident vertices by $\vertex$ and $\other{\vertex}$,
where $\vertex=\other{\vertex}$ if it is an $N$-loop. Analogously,
we let edge $\mate{\edge}$ of $\withoutDirichlet{(\mate{\graph})}$
have color $\mate{\colorEdge}$ and possibly identical incident vertices
$\mate{\vertex}$ and $\mate{\other{\vertex}}$. Then, 
\[
[\adjacencyMatrix{\colorEdge}{\graph}]_{\vertex\other{\vertex}}=[\adjacencyMatrix{\colorEdge}{\graph}]_{\other{\vertex}\vertex}=[\adjacencyMatrix{\colorEdge}{\mate{\graph}}]_{\mate{\vertex}\mate{\other{\vertex}}}=[\adjacencyMatrix{\colorEdge}{\mate{\graph}}]_{\mate{\other{\vertex}}\mate{\vertex}}=1
\]
are the only non-vanishing entries in their respective row and column.
In particular,
\[
[\transplantationMatrix]_{\mate{\vertex\vertex}}=[\adjacencyMatrix{\colorEdge}{\graph}\transplantationMatrix]_{\other{\vertex}\mate{\vertex}}=[\transplantationMatrix\adjacencyMatrix{\colorEdge}{\mate{\graph}}]_{\other{\vertex}\mate{\vertex}}=[\transplantationMatrix]_{\other{\vertex}\mate{\other{\vertex}}}\quad\text{and}\quad[\transplantationMatrix]_{\mate{\other{\vertex}\vertex}}=[\adjacencyMatrix{\colorEdge}{\graph}\transplantationMatrix]_{\mate{\vertex\vertex}}=[\transplantationMatrix\adjacencyMatrix{\colorEdge}{\mate{\graph}}]_{\mate{\vertex\vertex}}=[\transplantationMatrix]_{\mate{\vertex\other{\vertex}}}.
\]
Hence,
\[
[\transplantationMatrix]_{\vertex\mate{\vertex}}+[\transplantationMatrix]_{\mate{\vertex\other{\vertex}}}=[\transplantationMatrix]_{\mate{\other{\vertex}\vertex}}+[\transplantationMatrix]_{\other{\vertex}\mate{\other{\vertex}}}=[\transplantationMatrix]_{\mate{\vertex\vertex}}+[\transplantationMatrix]_{\mate{\other{\vertex}\vertex}}=[\transplantationMatrix]_{\mate{\vertex\other{\vertex}}}+[\transplantationMatrix]_{\mate{\other{\vertex}\other{\vertex}}},
\]
which allows to define a transplantation matrix for the line graphs
associated with $\graph$ and~$\mate{\graph}$.
\begin{defn}
The line graph transplantation matrix $\transplantationMatrixLineGraphs\in\mathbb{R}^{\nrVerticesLineGraph{}\times\nrVerticesLineGraph{}}$
coming from $\transplantationMatrix$ is given~by
\[
[\transplantationMatrixLineGraphs]_{\edge\mate{\edge}}=\begin{cases}
[\transplantationMatrix]_{\vertex\mate{\vertex}}+[\transplantationMatrix]_{\mate{\vertex\other{\vertex}}}=[\transplantationMatrix]_{\mate{\vertex\vertex}}+[\transplantationMatrix]_{\mate{\other{\vertex}\vertex}} & \text{if }\colorEdge=\mate{\colorEdge}\text{ and }\mate{\vertex}\neq\mate{\other{\vertex}},\\
{}[\transplantationMatrix]_{\vertex\mate{\vertex}}=[\transplantationMatrix]_{\mate{\vertex\other{\vertex}}}=[\transplantationMatrix]_{\mate{\other{\vertex}\vertex}}=[\transplantationMatrix]_{\mate{\other{\vertex}\other{\vertex}}} & \text{if }\colorEdge=\mate{\colorEdge}\text{ and }\mate{\vertex}=\mate{\other{\vertex}},\\
0 & \text{otherwise.}
\end{cases}
\]

\end{defn}
For later reference, we note that if $\edge$ was a $\colorEdge$-colored
$D$-loop, i.e., if $[\adjacencyMatrix{\colorEdge}{\graph}]_{\vertex\vertex}=-1$,
then
\[
[\transplantationMatrix]_{\mate{\vertex\vertex}}=-[\adjacencyMatrix{\colorEdge}{\graph}\transplantationMatrix]_{\vertex\mate{\vertex}}=-[\transplantationMatrix\adjacencyMatrix{\colorEdge}{\mate{\graph}}]_{\vertex\mate{\vertex}}=-[\transplantationMatrix]_{\vertex\mate{\other{\vertex}}},\quad\text{so that }[\transplantationMatrix]_{\vertex\mate{\vertex}}+[\transplantationMatrix]_{\mate{\vertex\other{\vertex}}}=0.
\]
Similarly, if $\mate{\edge}$ was a $\colorEdge$-colored $D$-loop,
i.e., if $[\adjacencyMatrix{\colorEdge}{\mate{\graph}}]_{\mate{\vertex}\mate{\vertex}}=-1$,
then
\[
[\transplantationMatrix]_{\mate{\vertex\vertex}}=[\adjacencyMatrix{\colorEdge}{\graph}\transplantationMatrix]_{\other{\vertex}\mate{\vertex}}=[\transplantationMatrix\adjacencyMatrix{\colorEdge}{\mate{\graph}}]_{\other{\vertex}\mate{\vertex}}=-[\transplantationMatrix]_{\other{\vertex}\mate{\vertex}},\quad\text{so that }[\transplantationMatrix]_{\mate{\vertex\vertex}}+[\transplantationMatrix]_{\other{\vertex}\mate{\vertex}}=0.
\]

\begin{lem}
The line graph transplantation matrix $\transplantationMatrixLineGraphs$
is invertible and satisfies
\begin{eqnarray*}
\adjacencyMatrix{\colorEdge}{\lineGraphVertexColored{\graph}}\transplantationMatrixLineGraphs & = & \transplantationMatrixLineGraphs\adjacencyMatrix{\colorEdge}{\lineGraphVertexColored{\mate{\graph}}}\quad\text{for every }\colorEdge\in\{1,2,\ldots,\nrColors\},\text{ as well as}\\
\adjacencyMatrix{\{\colorEdge,\other{\colorEdge}\}}{\lineGraphEdgeColored{\graph}}\transplantationMatrixLineGraphs & = & \transplantationMatrixLineGraphs\adjacencyMatrix{\{\colorEdge,\other{\colorEdge}\}}{\lineGraphEdgeColored{\mate{\graph}}}\quad\text{for every }\colorEdge,\other{\colorEdge}\in\{1,2,\ldots,\nrColors\}\text{ with }\colorEdge\neq\other{\colorEdge}.
\end{eqnarray*}
\end{lem}
\begin{proof}
For each $\colorEdge\in\{1,2,\ldots,\nrColors\}$, let $\setOfColoredEdges{\colorEdge}=\{\nrVerticesLineGraph 1+\nrVerticesLineGraph 2+\ldots+\nrVerticesLineGraph{\colorEdge-1}+\edge^{\colorEdge}\mid\edge^{\colorEdge}=1,2,\ldots,\nrVerticesLineGraph{\colorEdge}\}$,
which corresponds to the $\colorEdge$-colored edges of $\withoutDirichlet{\graph}$
as well as to the $\colorEdge$-colored edges of $\withoutDirichlet{(\mate{\graph})}$,
which in turn correspond to the $\colorEdge$-colored vertices of
$\lineGraphVertexColored{\graph}$ and $\lineGraphVertexColored{\mate{\graph}}$,
respectively. In~order to show that the block diagonal matrix $\transplantationMatrixLineGraphs$
is invertible, it suffices to show that each of its $\nrColors$ diagonal
blocks has linearly independent rows. Let $\colorEdge\in\{1,2,\ldots,\nrColors\}$,
and assume that $(a_{\edge})_{\edge\in\setOfColoredEdges{\colorEdge}}\in\mathbb{R}^{\nrVerticesLineGraph{\colorEdge}}$
satisfies 
\[
0=\sum_{\edge\in\setOfColoredEdges{\colorEdge}}a_{\edge}[\transplantationMatrixLineGraphs]_{\edge\mate{\edge}}\quad\text{for every }\mate{\edge}\in\setOfColoredEdges{\colorEdge}.
\]
As before, we let edge $\mate{\edge}\in\setOfColoredEdges{\colorEdge}$
of $\withoutDirichlet{(\mate{\graph})}$ have possibly identical incident
vertices $\mate{\vertex}$ and $\mate{\other{\vertex}}$. We~first
consider the case $\mate{\vertex}\neq\mate{\other{\vertex}}$. Recall
that every vertex~$\vertex$ of $\graph$ has exactly one incident
$\colorEdge$-colored edge, which we denote by $\edge(\vertex,\colorEdge)$.
Let
\[
\other a_{\edge(\vertex,\colorEdge)}=\begin{cases}
\frac{1}{2}a_{\edge(\vertex,\colorEdge)} & \text{if }\edge(\vertex,\colorEdge)\text{ is a link},\\
a_{\edge(\vertex,\colorEdge)} & \text{if }\edge(\vertex,\colorEdge)\text{ is an }N\text{-loop},\\
0 & \text{if }\edge(\vertex,\colorEdge)\text{ is a }D\text{-loop}.
\end{cases}
\]
If $\edge=\edge(\vertex,\colorEdge)$ is an $N$-loop at $\vertex$,
then 
\[
a_{\edge}[\transplantationMatrixLineGraphs]_{\edge\mate{\edge}}=\other a_{\edge(\vertex,\colorEdge)}([\transplantationMatrix]_{\vertex\mate{\vertex}}+[\transplantationMatrix]_{\vertex\mate{\other{\vertex}}}),
\]
whereas if $\edge=\edge(\vertex,\colorEdge)=\edge(\other{\vertex},\colorEdge)$
is a link between $\vertex$ and $\other{\vertex}$, then
\[
a_{\edge}[\transplantationMatrixLineGraphs]_{\edge\mate{\edge}}=2\other a_{\edge}[\transplantationMatrixLineGraphs]_{\edge\mate{\edge}}=\other a_{\edge(\vertex,\colorEdge)}([\transplantationMatrix]_{\vertex\mate{\vertex}}+[\transplantationMatrix]_{\mate{\vertex\other{\vertex}}})+\other a_{\edge(\other{\vertex},\colorEdge)}([\transplantationMatrix]_{\other{\vertex}\mate{\vertex}}+[\transplantationMatrix]_{\mate{\other{\vertex}\other{\vertex}}}).
\]
Hence,
\[
0=\sum_{\vertex=1}^{\nrVertices}\other a_{\edge(\vertex,\colorEdge)}([\transplantationMatrix]_{\vertex\mate{\vertex}}+[\transplantationMatrix]_{\mate{\vertex\other{\vertex}}})=\sum_{\vertex=1}^{\nrVertices}\other a_{\edge(\vertex,\colorEdge)}[\transplantationMatrix]_{\vertex\mate{\vertex}}+\sum_{\other{\vertex}=1}^{\nrVertices}\other a_{\edge(\other{\vertex},\colorEdge)}[\transplantationMatrix]_{\mate{\other{\vertex}\other{\vertex}}}=2\sum_{\vertex=1}^{\nrVertices}\other a_{\edge(\vertex,\colorEdge)}[\transplantationMatrix]_{\vertex\mate{\vertex}},
\]
where in the last equality we used that each vertex $\other{\vertex}$
is the unique $c$-neighbor of some vertex~$\vertex$, meaning $[\adjacencyMatrix{\colorEdge}{\graph}]_{\vertex\other{\vertex}}\neq0$,
in which case $\other a_{\edge(\other{\vertex},\colorEdge)}[\transplantationMatrix]_{\mate{\other{\vertex}\other{\vertex}}}=\other a_{\edge(\vertex,\colorEdge)}[\transplantationMatrix]_{\mate{\vertex\vertex}}$.
The same arguments apply if $\mate{\vertex}=\mate{\other{\vertex}}$,
except for the summands involving $\mate{\other{\vertex}}$ which
disappear. Since the rows of $\transplantationMatrix$ are linearly
independent, we deduce that $\other a_{\edge(\vertex,\colorEdge)}=0$
for every $\vertex\in\{1,2,\ldots,\nrVertices\}$. In other words,
$(a_{\edge})_{\edge\in\setOfColoredEdges{\colorEdge}}=0$, which proves
that $\transplantationMatrixLineGraphs$ is invertible.

Next, we show that for every $\other{\colorEdge}\in\{1,2,\ldots,\nrColors\}$
and $\edge,\mate{\edge}\in\{1,2,\ldots,\nrVerticesLineGraph{}\}$
\[
[\adjacencyMatrix{\other{\colorEdge}}{\lineGraphVertexColored{\graph}}\transplantationMatrixLineGraphs]_{\edge\mate{\edge}}=[\transplantationMatrixLineGraphs\adjacencyMatrix{\other{\colorEdge}}{\lineGraphVertexColored{\mate{\graph}}}]_{\edge\mate{\edge}}.
\]
Let $\edge\in\setOfColoredEdges{\colorEdge}$ and $\mate{\edge}\in\setOfColoredEdges{\mate{\colorEdge}}$.
Since $\transplantationMatrixLineGraphs$ is block diagonal, we have
\[
[\adjacencyMatrix{\other{\colorEdge}}{\lineGraphVertexColored{\graph}}\transplantationMatrixLineGraphs]_{\edge\mate{\edge}}=\sum_{\other{\edge}\in\setOfColoredEdges{\mate{\colorEdge}}}[\adjacencyMatrix{\other{\colorEdge}}{\lineGraphVertexColored{\graph}}]_{\edge\other{\edge}}[\transplantationMatrixLineGraphs]_{\other{\edge}\mate{\edge}}\quad\text{and}\quad[\transplantationMatrixLineGraphs\adjacencyMatrix{\other{\colorEdge}}{\lineGraphVertexColored{\mate{\graph}}}]_{\edge\mate{\edge}}=\sum_{\mate{\other{\edge}}\in\setOfColoredEdges{\colorEdge}}[\transplantationMatrixLineGraphs]_{\edge\mate{\other{\edge}}}[\adjacencyMatrix{\other{\colorEdge}}{\lineGraphVertexColored{\mate{\graph}}}]_{\mate{\other{\edge}}\mate{\edge}}.
\]
If $\other{\colorEdge}\neq\colorEdge$ or $\other{\colorEdge}=\colorEdge=\mate{\colorEdge}$,
then $[\adjacencyMatrix{\other{\colorEdge}}{\lineGraphVertexColored{\graph}}]_{\edge\other{\edge}}=0$
for all $\other{\edge}\in\setOfColoredEdges{\mate{\colorEdge}}$,
and $[\adjacencyMatrix{\other{\colorEdge}}{\lineGraphVertexColored{\mate{\graph}}}]_{\mate{\other{\edge}}\mate{\edge}}=0$
for all $\mate{\other{\edge}}\in\setOfColoredEdges{\colorEdge}$.
It therefore suffices to consider the case $\other{\colorEdge}=\colorEdge\neq\mate{\colorEdge}$.
As before, we let edge $\edge$ of $\withoutDirichlet{\graph}$ and
edge $\mate{\edge}$ of $\withoutDirichlet{(\mate{\graph})}$ have
incident vertices $\{\vertex,\other{\vertex}\}$ and $\{\mate{\vertex},\mate{\other{\vertex}}\}$,
respectively. Note that each of the sums above has at most $2$ non-vanishing
summands, which correspond to the $\mate{\colorEdge}$-colored edges
at $\vertex$ and $\other{\vertex}$, as well as the $\colorEdge$-colored
edges at $\mate{\vertex}$ and $\mate{\other{\vertex}}$, respectively.
Regardless of whether these edges are links, $N$-loops, or $D$-loops,
if $\mate{\vertex}\neq\mate{\other{\vertex}}$, then 
\[
[\adjacencyMatrix{\other{\colorEdge}}{\lineGraphVertexColored{\graph}}\transplantationMatrixLineGraphs]_{\edge\mate{\edge}}=[\transplantationMatrix]_{\vertex\mate{\vertex}}+[\transplantationMatrix]_{\mate{\vertex\other{\vertex}}}+[\transplantationMatrix]_{\other{\vertex}\mate{\vertex}}+[\transplantationMatrix]_{\mate{\other{\vertex}\other{\vertex}}}=[\transplantationMatrixLineGraphs\adjacencyMatrix{\other{\colorEdge}}{\lineGraphVertexColored{\mate{\graph}}}]_{\edge\mate{\edge}},
\]
whereas if $\mate{\vertex}=\mate{\other{\vertex}}$, then
\[
[\adjacencyMatrix{\other{\colorEdge}}{\lineGraphVertexColored{\graph}}\transplantationMatrixLineGraphs]_{\edge\mate{\edge}}=[\transplantationMatrix]_{\vertex\mate{\vertex}}+[\transplantationMatrix]_{\other{\vertex}\mate{\vertex}}=[\transplantationMatrixLineGraphs\adjacencyMatrix{\other{\colorEdge}}{\lineGraphVertexColored{\mate{\graph}}}]_{\edge\mate{\edge}}.
\]

Finally, we show that for every $\colorEdge,\other{\colorEdge}\in\{1,2,\ldots,\nrColors\}$
with $\colorEdge\neq\other{\colorEdge}$ and $\edge\in\{1,2,\ldots,\nrVerticesLineGraph{}\}$
\[
[\adjacencyMatrix{\{\colorEdge,\other{\colorEdge}\}}{\lineGraphEdgeColored{\graph}}\transplantationMatrixLineGraphs]_{\edge}=[\transplantationMatrixLineGraphs\adjacencyMatrix{\{\colorEdge,\other{\colorEdge}\}}{\lineGraphEdgeColored{\mate{\graph}}}]_{\edge},
\]
where we reused the notation $[M]_{m}$ for row~$m$ of the matrix
$M$. We use the same idea as above and note that if $\edge\notin\setOfColoredEdges{\colorEdge}\cup\setOfColoredEdges{\other{\colorEdge}}$,
then $[\adjacencyMatrix{\{\colorEdge,\other{\colorEdge}\}}{\lineGraphEdgeColored{\graph}}]_{\edge}=0$,
i.e., $[\adjacencyMatrix{\{\colorEdge,\other{\colorEdge}\}}{\lineGraphEdgeColored{\graph}}\transplantationMatrixLineGraphs]_{\edge}=0$,
and $[\transplantationMatrixLineGraphs]_{\edge\mate{\other{\edge}}}\neq0$
only if $\mate{\other{\edge}}\notin\setOfColoredEdges{\colorEdge}\cup\setOfColoredEdges{\other{\colorEdge}}$,
i.e., $[\transplantationMatrixLineGraphs\adjacencyMatrix{\{\colorEdge,\other{\colorEdge}\}}{\lineGraphEdgeColored{\mate{\graph}}}]_{\edge}=0$.
If $\edge\in\setOfColoredEdges{\colorEdge}$, then $[\adjacencyMatrix{\{\colorEdge,\other{\colorEdge}\}}{\lineGraphEdgeColored{\graph}}]_{\edge}=[\adjacencyMatrix{\colorEdge}{\lineGraphVertexColored{\graph}}]_{\edge}$,
i.e., $[\adjacencyMatrix{\{\colorEdge,\other{\colorEdge}\}}{\lineGraphEdgeColored{\graph}}\transplantationMatrixLineGraphs]_{\edge}=[\adjacencyMatrix{\colorEdge}{\lineGraphVertexColored{\graph}}\transplantationMatrixLineGraphs]_{\edge}$,
and $[\transplantationMatrixLineGraphs]_{\edge\mate{\other{\edge}}}\neq0$
only if $\mate{\other{\edge}}\in\setOfColoredEdges{\colorEdge}$ in
which case we have $[\adjacencyMatrix{\{\colorEdge,\other{\colorEdge}\}}{\lineGraphEdgeColored{\mate{\graph}}}]_{\mate{\other{\edge}}}=[\adjacencyMatrix{\colorEdge}{\lineGraphVertexColored{\mate{\graph}}}]_{\mate{\other{\edge}}}$,
i.e., $[\transplantationMatrixLineGraphs\adjacencyMatrix{\{\colorEdge,\other{\colorEdge}\}}{\lineGraphEdgeColored{\mate{\graph}}}]_{\edge}=[\transplantationMatrixLineGraphs\adjacencyMatrix{\colorEdge}{\lineGraphVertexColored{\mate{\graph}}}]_{\edge}$.
Similarly, if $\edge\in\setOfColoredEdges{\other{\colorEdge}}$, then
\[
[\adjacencyMatrix{\{\colorEdge,\other{\colorEdge}\}}{\lineGraphEdgeColored{\graph}}\transplantationMatrixLineGraphs]_{\edge}=[\adjacencyMatrix{\other{\colorEdge}}{\lineGraphVertexColored{\graph}}\transplantationMatrixLineGraphs]_{\edge}=[\transplantationMatrixLineGraphs\adjacencyMatrix{\other{\colorEdge}}{\lineGraphVertexColored{\mate{\graph}}}]_{\edge}=[\transplantationMatrixLineGraphs\adjacencyMatrix{\{\colorEdge,\other{\colorEdge}\}}{\lineGraphEdgeColored{\mate{\graph}}}]_{\edge}.
\]

\end{proof}
\bibliographystyle{amsalpha}
\bibliography{line_graphs}

\end{document}